\documentclass[11pt]{article}
\usepackage{amssymb}
\usepackage{amsmath}
\input{epsf.sty}  
\usepackage{color}

\newtheorem{theorem}{Theorem} 
\newtheorem{remark}{Remark}
\newtheorem{definition}{Definition}

\newtheorem{example}{Example} 
\newtheorem{proposition}{Proposition} 
\newtheorem{corollary}{Corollary} 
\newtheorem{counterexample}{Counterexample} 

\newcommand{\barF}{\overline{F}}

\newenvironment{proof}{\begin{trivlist}
\item[\hspace{\labelsep}{\bf\noindent Proof. }]}
{$\hfill\Box$\end{trivlist}}
\title{\huge\bf 
A QUANTILE-BASED PROBABILISTIC \\
MEAN VALUE THEOREM
}
\author{
{\bf Antonio  Di Crescenzo}\footnote{
Dipartimento di Matematica, Universit\`a di Salerno, 
Via Giovanni Paolo II, 132; 84084 Fisciano (SA), Italy; 
email: adicrescenzo@unisa.it}, 
\quad 
{\bf Barbara Martinucci}\footnote{
Dipartimento di Matematica, Universit\`a di Salerno, 
Via Giovanni Paolo II, 132; 84084 Fisciano (SA), Italy; 
email: bmartinucci@unisa.it}, 
\quad  
{\bf Julio Mulero}\footnote{
Departamento de Estad\'{\i}stica e Investigaci\'on Operativa, Universidad de Alicante,
Apartado de Correos, 99; 03080 Alicante, Spain; 
email: julio.mulero@ua.es
}}
\begin{document}
%
\maketitle
\begin{abstract}
For nonnegative random variables with finite means we introduce an analogous of the equilibrium 
residual-lifetime distribution based on the quantile function. This allows to construct new 
distributions with support $(0,1)$, and to obtain a new quantile-based version of the probabilistic 
generalization of Taylor's theorem. Similarly, for pairs of stochastically ordered random variables 
we come to a new quantile-based form of the probabilistic mean value theorem.
The latter involves a distribution that generalizes the Lorenz curve. We investigate the special 
case of proportional quantile functions and apply the given results to various models based on 
classes of distributions and measures of risk theory. Motivated by some stochastic comparisons, 
we also introduce the `expected reversed proportional shortfall order', and a new 
characterization of random lifetimes involving the reversed hazard rate function. 

\smallskip\noindent
{\em Short title:\/} A quantile-based probabilistic mean value theorem.

\end{abstract}
\section{Introduction}
The quantile function, being the inverse of the cumulative distribution function of a random variable, is 
often invoked in applied probability and statistics. In certain cases the approach based on quantile functions 
is more fruitful than the use of cumulative distribution functions, since quantile functions are less influenced 
by extreme statistical observations. For instance, quantile functions can be properly employed to formulate 
properties of entropy function and other information measures for nonnegative absolutely continuous 
random variables (see  Sunoj and Sankaran \cite{SuSa12} and Sunoj {\em et al.}\ \cite{SuSa13}). 
They are also employed in problems that ask for comparisons based on variability stochastic orders such 
as the dilation order, the dispersive order (see Shaked and Shanthikumar \cite{ShSh07}) or the TTT 
transform order (cf.\ Kochar {\em et al}.\ \cite{KoSh2002}). In addition, several notions of risk theory 
and mathematical finance are expressed in terms of quantile functions (see, for instance, 
Belzunce {\em et al.}\ \cite{BePiRuSo12} and \cite{BePiRuSo13}). 
\par 
In this paper we use the quantile functions in order 
to build some stochastic models and obtain various results involving distributions with support $(0,1)$. 
We are motivated by previous researches in which the equilibrium distribution of nonnegative 
random variables plays a key role and allows to obtain probabilistic generalizations of Taylor's 
theorem (see \cite{Li94} and \cite{MaWh93}) and of the  mean value theorem (see \cite{DiCr99}). 
\par
In Section \ref{section:prel} we present some preliminary notions on 
quantile function and Lorenz curve. Then, in Section \ref{section:Taylor} we obtain a probabilistic 
generalization of TaylorÕs theorem based on a suitably defined `quantile analogue' of the equilibrium 
distribution, whose density is an extension of the Lorenz curve based on stochastically ordered 
random variables. Moreover, such distribution is involved in a quantile-based version of the 
probabilistic mean value theorem provided in Section \ref{section:mvt}. 
A special case dealing with proportional quantile functions is also discussed. 
Finally, various examples of applications are considered in Section \ref{section:appl}: 
the first involves typical classes of distributions (NBU and IFR notions) and  conditional value-at-risks; 
the second involves concepts of risk theory, as the proportional conditional value-at-risk; the third and 
the fourth applications are founded on distribution functions defined as suitable ratios of quantile functions, 
and involve the notion of average value-at-risk. 
\par
We point out that, aiming to obtain useful stochastic comparisons, in this paper we introduce two new 
concepts that deserve interest in the field of stochastic orders and characterizations of distributions. 
In Section \ref{section:mvt} we propose the `expected reversed proportional shortfall order', 
which is dual to a recently proposed stochastic order. In Section \ref{app3} we provide a new 
characterization of random lifetimes, expressed by stating that $x\,\tau(x)$ is decreasing 
for $x>0$, where $\tau(x)$ is the reversed hazard rate function. 
\par
Throughout the paper, $[X\,|\,B]$ denotes a random variable having the same distribution as $X$ conditional 
on $B$, the terms decreasing and increasing are used in non-strict sense, and $g'$ denotes the derivative 
of $g$. 
\section{Preliminary notions}\label{section:prel}
Given a  random variable $X$, let us denote its distribution function by $F(x)=P(X\leq x)$, 
$x\in\mathbb{R}$, and its complementary distribution function by $\overline F(x)=1-F(x)$, $x\in\mathbb{R}$. 
The quantile function of $X$, when existing, is given by 
\begin{equation}
 Q(u)=\inf\{x\in\mathbb{R} : F(x)\geq u\}, 
 \qquad 0<u<1.
 \label{eq:defQu}
\end{equation}
Moreover, if $Q(u)$ is differentiable, the quantile density function of $X$ is given by  
\begin{equation}
 q(u)=Q'(u), \qquad 0<u<1.  
 \label{eq:defdensqu}
\end{equation}
\begin{definition}
We denote by ${\cal D}$ the family of all absolutely continuous random variables with finite mean  
such that the quantile function (\ref{eq:defQu}) satisfies $Q(0)=0$, and the quantile density 
function (\ref{eq:defdensqu}) exists. 
\end{definition}
A random variable in ${\cal D}$ is thus nonnegative and may represent a distribution of interest 
in actuarial applications or in risk theory, such as an income or a loss. 
If $X\in {\cal D}$, it has finite nonzero mean, and the function 
\begin{equation}
 L(p)=\frac{1}{E[X]}\int_0^p Q(u)\,{\rm d}u, \qquad 0\leq p\leq 1 
 \label{eq:defLp}
\end{equation}
denotes the Lorenz curve of $X$. If the individuals of a given population share a common good such 
as wealth, which is distributed according to $X$, then $L(p)$ gives the cumulative share 
of individuals, from the lowest to the highest, owing the fraction $p$ of the common good. 
Hence, $L(p)$ is often used in insurance to describe the inequality among the incomes of individuals. 
See, for instance, Singpurwalla and Gordon \cite{SiGo12} and Shaked and Shanthikumar \cite{ShSh98} 
for various applications of the Lorenz curve and its connections with stochastic orders.  
\par
It is well known that (\ref{eq:defLp}) is the distribution function of an absolutely continuous random 
variable, say $X^L$, taking values in $(0,1)$. In the following proposition we express the mean of an 
arbitrary function of $X^L$ in terms of the quantile function (\ref{eq:defQu}). To this aim we recall 
that if $g:(0,+\infty)\to\mathbb{R}$ is an integrable function then, for all $0\leq p_1<p_2\leq 1$,
\begin{equation}
 \frac{1}{p_2-p_1}\int_{p_1}^{p_2} g(Q(u))\,{\rm d}u
 =E[g(X)\,|\,Q(p_1)<X\leq Q(p_2)]. 
 \label{eq:integrgQ}
\end{equation}
\begin{proposition}\label{prop:EhXL}
Let $X\in {\cal D}$ and let $X^L$ have distribution function $(\ref{eq:defLp})$. If $h:(0,1)\to \mathbb{R}$ 
is such that $h\cdot Q$ is integrable in $(0,1)$, then  
\begin{equation}
 E[h(X^L)]=\frac{1}{E[X]}\int_0^1 h(u)\,Q(u)\,{\rm d}u 
 =\frac{1}{E[X]}\,E[h(F(X))\,X]
 \label{eq:medhXL}
\end{equation}
or, equivalently,
$$
 E[h(X^L)]=\frac{1}{E[Q(U)]}\,E[h(U)\,Q(U)], 
$$
where $U=F(X)$ is uniformly distributed in $(0,1)$. 
\end{proposition}
\begin{proof}
Since $X^L$ has distribution function $(\ref{eq:defLp})$, the proof follows from identity 
$F[Q(u)]=u$, $0<u<1$, and from Eq.\ (\ref{eq:integrgQ}) for $p_1=0$ and $p_2=1$. 
\end{proof}
\par
As an immediate application of Proposition \ref{prop:EhXL} we have that the moments of $X^L$, 
when existing, are given by:
\begin{equation}
 E[(X^L)^k]=\frac{1}{E[Q(U)]}\,E[U^k\,Q(U)], \qquad k=1,2,\ldots. 
 \label{eq:EXLk}
\end{equation}
\begin{remark}\rm 
Let $F(x)=x^{\alpha}$, $0\leq x\leq 1$, $\alpha>0$. Then, $X$ is identically distributed to $X^L$ if, 
and only if, $\alpha$ is equal to the reciprocal of the golden number, i.e.\ $\alpha=(-1+\sqrt{5})/2$. 
\end{remark}
%
\section{An analogous of Taylor's theorem}\label{section:Taylor}
It is well-known that the equilibrium distribution arises as the limiting distribution of the forward recurrence 
time in a renewal process. Its role in applied contexts has been largely investigated (see, as example, 
Gupta \cite{Gu2007} and references therein). For instance, we recall that the iterates of equilibrium 
distributions have been used \\
- to characterize family of distributions (see Unnikrishnan Nair and Preeth \cite{NaPr2009}), \\
- to construct sequences of stochastic orders (see Fagiuoli and Pellerey \cite{FaPe93}), \\
- to determine properties related to the moments of random variables of interest in risk theory 
(see  Lin and Willmot \cite{LiWi2000}).  
\par
Moreover, a probabilistic generalization of Taylor's theorem (studied by Massey and Whitt \cite{MaWh93} 
and Lin \cite{Li94}) allows to express the expectation of a functional of  random variable in terms of 
suitable expectations involving the iterates of its equilibrium distribution. 
\par
We recall that for a random variable $X\in {\cal D}$ the density 
of the equilibrium distribution of $X$ and the density of $X^L$ are given respectively by 
\begin{equation}
 f_{X_e}(x)=\frac{\barF(x)}{E[X]}, 
 \quad 0< x<+\infty,
 \qquad 
 f_{X^L}(u)=\frac{Q(u)}{E[X]}, 
 \quad 0< u<1. 
 \label{eq:densXeXL}
\end{equation}
On the ground of the analogy between such densities, in this section we obtain an analogous of the 
probabilistic generalization of Taylor's theorem which involves $X^L$.  
\begin{theorem}\label{theor:1}
Let $X\in{\cal D}$;  if $g:(0,1)\to \mathbb{R}$ is a differentiable function such that $g'\cdot Q$ 
is integrable on $(0,1)$, then  
\begin{equation}
 E[\{g(1)-g(U)\}\,q(U)]=E\left[g'(X^L)\right]E[X],
 \label{eq:theor1}
\end{equation}
where $U$ is uniformly distributed in $(0,1)$. 
\end{theorem}
\begin{proof}
From the mean value theorem, Eq.\ (\ref{eq:defdensqu}), and condition $Q(0)=0$ we have 
\begin{equation*}
 \begin{split}
 E[\{g(1)-g(U)\}\,q(U)] &= E\left[q(U) \int_0^1 g'(u)\,{\bf 1}_{\{U\leq u\}}\,{\rm d}u\right] \\
 &= \int_0^1 g'(u)\, E[q(U)\,{\bf 1}_{\{U\leq u\}}] \,{\rm d}u \\
 &= \int_0^1 g'(u)\, Q(u) \,{\rm d}u. 
 \end{split}
\end{equation*}
The proof of (\ref{eq:theor1}) then follows from the first equality in (\ref{eq:medhXL}). 
\end{proof}
\begin{example}\rm 
Let $X$ have Lomax distribution, with quantile function $Q(u)=\lambda[(1-u)^{-1/\alpha}-1]$, 
$0<u<1$, and mean $E[X]=\lambda/(\alpha-1)$, for $\alpha>1$ and $\lambda>0$. Then, 
under the assumptions of Theorem \ref{theor:1} we have 
$$
 E[\{g(1)-g(U)\}\,(1-U)^{-1-1/\alpha}]=E\left[g'(X^L)\right]\,\frac{\alpha}{\alpha-1},
$$ 
where $X^L$ has density $f_{X^L}(u)=(\alpha-1)[(1-u)^{-1/\alpha}-1]$, $0<u<1$. 
\end{example}
\par
Hereafter we extend the result of Theorem \ref{theor:1} to a more general case, in which the 
right-hand-side of (\ref{eq:theor1}) is expressed in an alternative way.  
Let $g^{(n)}$ denote the $n$-th derivative of  $g$, for $n\geq 1$, and let $g^{(0)}=g$.  
\begin{theorem}\label{theor:2}
Let $X\in{\cal D}$;  if  $g:(0,1)\to\mathbb{R}$ is $n$-times differentiable and  
$g^{(n)}\cdot Q$ is integrable on $(0,1)$,  for any $n\geq 1$, then  
\begin{equation}
\begin{split}
 E[\{g(1)-g(U)\}\,q(U)]
 & =\sum_{k=1}^{n-1} \frac{1}{k!}E\left[g^{(k)}(U)(1-U)^kq(U)\right] \\
 & +\frac{1}{(n-1)!}E\left[g^{(n)}(X^L)(1-X^L)^{n-1}\right]E[X],
 \end{split}
 \label{eq:theorn}
\end{equation}
where $U$ is uniformly distributed in $(0,1)$. 
\end{theorem}
\begin{proof}
For $n=1,2,\ldots$, consider the function
\begin{equation}
 R_ng(u):=g(1)-\sum_{k=0}^{n-1} \frac{g^{(k)}(u)}{k!}\,(1-u)^k, \qquad 0<u<1.
 \label{eq:defRngxu}
\end{equation}
It has the following properties, for fixed $u\in(0,1)$:
$$ 
 R_1g(u)=g(1)-g(u),
$$
\begin{equation}
 \frac{\partial}{\partial u}R_ng(u)=- \frac{g^{(n)}(u)}{(n-1)!}\,(1-u)^{n-1}, 
 \qquad n\geq 1.
 \label{eq:propderivR}
\end{equation}
Applying Theorem \ref{theor:1} to the function $g^*(u):=R_ng(u)$ we have 
$$
 E[\{R_ng(U)-R_ng(1)\}\,q(U)] =-E\left[\frac{\partial}{\partial u}R_ng(u)\big|_{u=X^L}\right]E[X]. 
$$ 
Hence, noting that $R_ng(1)=0$ and making use of Eq.\ (\ref{eq:propderivR}) we obtain
$$
  E[R_ng(U)\,q(U)]  = \frac{1}{(n-1)!}\,E\left[g^{(n)}(X^L)(1-X^L)^{n-1}\right]\,E[X]. 
$$
Finally, substituting (\ref{eq:defRngxu}) in the left-hand-side and rearranging the terms, 
Eq.\  (\ref{eq:theorn}) follows. 
\end{proof}
\par
We note that Eq.\ (\ref{eq:theorn}) reduces to (\ref{eq:theor1}) when $n=1$. 
\begin{corollary}\label{coroll:1}
Under the assumptions of Theorem \ref{theor:2}, for $\alpha>0$ we have 
\begin{equation}
\begin{split}
 E[(1-U^{\alpha})\,q(U)]
 & =\sum_{k=1}^{n-1} {\alpha \choose k}E\left[U^{\alpha-k}(1-U)^kq(U)\right] \\
 & +n {\alpha \choose n}E\left[(X^L)^{\alpha-n}(1-X^L)^{n-1}\right]E[X],
 \end{split}
 \nonumber
\end{equation}
where ${\alpha \choose k}=\alpha(\alpha-1)(\alpha-2)\cdots (\alpha-k+1)/k!$ is the generalized 
binomial coefficient. 
\end{corollary}
\begin{proof}
The proof follows from Theorem \ref{theor:2} by setting $g(u)=u^{\alpha}$, $0<u<1$. 
\end{proof}
\par
For instance, making use of  Corollary \ref{coroll:1}, when $\alpha=1$ we have 
$$
 E[(1-U)\,q(U)]=E[X]=E[Q(U)],
$$
whereas when $\alpha=2$ the following indentities follow from Eq.\ (\ref{eq:EXLk}):
$$
 E[(1-U^2)\,q(U)]=2 E[X^L]\,E[X]=2 E[U\,Q(U)],
$$
$$
 E[(1-U)^2\,q(U)]=2 E[1-X^L]\,E[X]=2 E[(1-U)\,Q(U)].
$$
\section{An analogous of the mean value theorem}\label{section:mvt}
A suitable transformation investigated in Section 3 of Di Crescenzo \cite{DiCr99} allows to construct 
new probability densities  via differences of (complementary) distribution functions of two 
stochastically ordered random variables. We now aim to construct similarly  new 
densities from differences of quantile functions, thus extending (\ref{eq:defLp}) to a more general case. 
This allows us to obtain a quantile-based version of the probabilistic mean value theorem, 
given in Theorem \ref{theor:3} below. 
\par
Let us recall some useful definitions of stochastic orders  (see, for instance, 
Shaked and Shanthikumar \cite{ShSh07}). To this purpose, we denote by $Q_X(u)$ and $Q_Y(u)$ the 
quantile functions of two random variables $X$ and $Y$, defined as in Eq.\ (\ref{eq:defQu}). 
We say that $X$ is smaller than $Y$ 
\\
$\bullet$ \ in the usual stochastic order (denoted by $X \leq_{st} Y$) if $E[\phi(X)] \leq  E[\phi(Y)]$ 
for all non-decreasing functions $\phi$ for which the expectations exist, or equivalently if 
$\barF_X(t)\leq \barF_Y(t)$ for all $t\in\mathbb{R}$, i.e.\ $Q_X(u)\leq Q_Y(u)$ for all $0<u<1$; 
 \\
$\bullet$ \ in the hazard rate order (denoted by $X \leq_{hr} Y$) if
$\barF_X(t)/\barF_Y(t)$ is decreasing in $t$; 
\\
$\bullet$ \ in the reversed hazard rate order (denoted by $X \leq_{rh} Y$) if
$F_X(t)/F_Y(t)$ is decreasing in $t$; 
\\
$\bullet$ \ in the likelihood ratio order (denoted by $X \leq_{lr} Y$) if
$f_X(t)/f_Y(t)$ is decreasing in $t$;
\\
$\bullet$ \ in the star order (denoted by $X \leq_{*} Y$) if $Q_X(u)/Q_Y(u)$ is decreasing in $u\in(0,1)$ 
(see Section 4.B of Shaked and Shanthikumar \cite{ShSh07}); 
\\
$\bullet$ \ in the expected proportional shortfall order (denoted by $X \leq_{PS} Y$) if 
$\int_u^1 Q_X(p)\,{\rm d}p/\int_u^1 Q_Y(p)\,{\rm d}p$ is decreasing in $u\in(0,1)$ (see 
Belzunce {\em et al.}\ \cite{BePiRuSo12} for some equivalent conditions for this order).
\par
The following result is analogous to Proposition 3.1 of Di Crescenzo \cite{DiCr99}. 
\begin{proposition}\label{prop:1}
Let $X$ and $Y$ be random variables taking values in $\mathbb{R}$, and such that 
$-\infty<E[X]<E[Y]<+\infty$. Then
\begin{equation}
 f_{Z^L}(u)=\frac{Q_Y(u)-Q_X(u)}{E[Y]-E[X]}, 
 \qquad 0< u<1
 \label{eq:deffZLu}
\end{equation}
is the probability density function of an absolutely continuous random variable $Z^L$ taking values 
in $(0,1)$ if, and only if, $X\leq_{st}Y$. 
\end{proposition}
\begin{proof}
The proof immediately follows from the definition of the usual stochastic order. 
\end{proof}
\par
We remark that, due to $(\ref{eq:densXeXL})$, the densities $f_{X_e}$ and $f_{X^L}$ are 
respectively  monotonic decreasing and increasing, whereas density (\ref{eq:deffZLu}) is not 
necessarily monotonic. We also note that $f_{Z^L}$ can be expressed as a linear combination 
of two densities. Indeed, under the assumptions of Proposition \ref{prop:1}, for $0< u<1$ we have
\begin{equation}
 f_{Z^L}(u)=c\, f_{Y^L}(u)+(1-c)\,f_{X^L}(u), 
 \qquad 
 c=\frac{E[Y]}{E[Y]-E[X]}. 
 \label{eq:comb}
\end{equation}
Note that $c$ can be negative, in particular $c<1$ if and only if $E[X]<0$, and  
$c>1$ if and only if $E[X]>0$. 
\begin{remark}\rm 
Let 
$$
 L_{X,Y}(p)= \frac{1}{E[Y]-E[X]}\int_0^p[Q_Y(u)-Q_X(u)]{\rm d}u,\qquad 0\leq p\leq 1 
$$
be the distribution function corresponding to density (\ref{eq:deffZLu}). This is a suitable extension of the Lorenz 
curve (\ref{eq:defLp}). Indeed, assume that the individuals of a certain population share a common good 
such as wealth, distributed according to $Y$. Suppose that the income received by the individuals is subject 
to losses due to various reasons (e.g.\ taxes, faults, damages, etc.), distributed according to $X$. 
Hence, $Q_Y(u)-Q_X(u)$ is the {\em net\/} income received by the poorest fraction $u$ of the population, 
and thus $L_{X,Y}(p)$ gives the portion of the {\em net\/} wealth held by a portion $p$ of the population. 
Note that $Q_Y(u)-Q_X(u)$ is not necessarily increasing, and thus generally it is not a quantile function. 
\end{remark}
\par
Let us now introduce the operator $\Psi^L$ on the set of all pairs of random variables $X$ and $Y$ 
defined as in Proposition \ref{prop:1}, such that $\Psi^L(X,Y)$ denotes an absolutely continuous 
random variable having density (\ref{eq:deffZLu}). Hence, the random variable $X^L=\Psi^L(0,X)$ 
has distribution function (\ref{eq:defLp}), and the distribution of $Y^L=\Psi^L(0,Y)$ is similarly defined. 
\begin{example}\label{ex:1}\rm
(i) Let $X$ and $Y$ be exponentially distributed with parameters $\lambda_X$ and $\lambda_Y$, 
respectively, with $\lambda_X>\lambda_Y$. 
Then,  $\Psi^L(X,Y)$ is exponentially distributed with parameter 1. \\ 
(ii) Let $X$ and $Y$ be uniformly distributed in $(0,\alpha_X)$ and $(0,\alpha_Y)$, respectively, 
with $\alpha_X<\alpha_Y$. 
Then,  $\Psi^L(X,Y)$ is uniformly distributed in $(0,1)$.
\end{example}
\par
Aiming to focus on some stochastic comparisons, we now introduce a new stochastic 
order based on the quantile function, which is dual to the expected proportional shortfall order. 
\begin{definition}
We say that $X$ is smaller than $Y$ in the {\em expected reversed proportional shortfall order} 
(denoted by $X \leq_{RPS} Y$) if $\int_0^u Q_X(p)\,{\rm d}p/\int_0^u Q_Y(p)\,{\rm d}p$ 
is decreasing in $u\in(0,1)$. 
\end{definition}
Results and properties of such an order go beyond the scope of this article, and thus will be 
the object of future investigation. The proof of the following results follows from the definitions 
of the involved notions and some straightforward calculations, and thus is omitted.  
\begin{proposition} 
Under the assumptions of Proposition \ref{prop:1}, for $X^L=\Psi^L(0,X)$,  $Y^L=\Psi^L(0,Y)$ 
and $Z^L=\Psi^L(X,Y)$ we have: 
\\ 
(i) If $X \leq_{*} Y$, then $X^L \leq_{lr} Z^L$ and $Y^L \leq_{lr} Z^L$. 
\\
(ii) If $X \leq_{PS} Y$, then $X^L \leq_{hr} Z^L$ and $Y^L \leq_{hr} Z^L$.
\\
(iii) If $X \leq_{RPS} Y$, then $X^L \leq_{rh} Z^L$ and $Y^L \leq_{rh} Z^L$.   
\\
(iv) The following conditions are equivalent:\\
$\bullet$ \  $X^L\leq_{st}Y^L$, \\
$\bullet$ \  $X^L\leq_{st}Z^L$, \\
$\bullet$ \  $Y^L\leq_{st}Z^L$.  
\end{proposition}
\par
According to (\ref{eq:defdensqu}), hereafter $q_X$ and $q_Y$ denote respectively the quantile density 
functions of $X$ and $Y$. The next result can be viewed as a quantile-based analogue of the probabilistic 
mean value theorem given in Theorem 4.1 of Di Crescenzo \cite{DiCr99}. 
\begin{theorem}\label{theor:3}
Let $X,Y\in {\cal D}$ and such that $X\leq_{st}Y$. 
Moreover, let $g:(0,1)\to \mathbb{R}$ be a differentiable function, and let 
$g'\cdot Q_X$ and $g'\cdot Q_Y$ be integrable on $(0,1)$. Then, for $Z^L=\Psi^L(X,Y)$ we 
have that $E\left[g'(Z^L)\right]$ is finite, and 
\begin{equation}
 E[\{g(1)-g(U)\}\,\{q_Y(U)-q_X(U)\}]=E\left[g'(Z^L)\right]\{E[Y]-E[X]\},
 \label{eq:theor3}
\end{equation}
where $U$ is uniformly distributed in $[0,1]$. 
\end{theorem}
\begin{proof}
From Theorem \ref{theor:1} and Eq.\ (\ref{eq:comb}) we obtain 
\[
\begin{split}
 E[\{g(1)-g(U)\}\,\{q_Y(U)-q_X(U)\}] &=  E\left[g'(Y^L)\right] E[Y]-E\left[g'(X^L)\right] E[X] \\
 &= \left\{c \,E\left[g'(Y^L)\right]+(1-c)\,E\left[g'(X^L)\right]\right\} \{E[Y]-E[X]\},
\end{split}
\nonumber 
\]
with $c=E[Y]/(E[Y]-E[X])$. This immediately gives Eq.\ (\ref{eq:theor3}). 
\end{proof}
\par
As example, under the assumptions of Theorem \ref{theor:3}, for $g(u)=u^{\alpha}$, $0< u< 1$,  
we have: 
$$
 E[(1-U^{\alpha})\,\{q_Y(U)-q_X(U)\}]=\alpha\,E\left[(Z^L)^{\alpha-1}\right]\{E[Y]-E[X]\}, 
 \qquad \alpha>0.
$$
In particular, when $\alpha=1$ we get the indentity 
$$
 E[(1-U)\,\{q_Y(U)-q_X(U)\}]=E[Y]-E[X], 
$$
which does not depend on $Z^L$. 
\subsection{Proportional quantile functions}
Let $X,Y\in {\cal D}$ have proportional quantile functions. For instance 
(see Escobar and Meeker \cite{EsMe2006}) such assumption leads 
to a scale-accelerated failure-time model. Let 
\begin{equation}
 Q_X(u)=\alpha_X\,\varphi(u), \qquad 
 Q_Y(u)=\alpha_Y\, \varphi(u), \qquad 
 \forall u\in(0,1),  
 \label{eq:propquant}
\end{equation}
with $\alpha_X>0$, $\alpha_Y>0$, where $\varphi(u)$ is a suitable increasing and differentiable 
function such that $\eta:=\int_0^1 \varphi(u)\,{\rm d}u$ is finite and $\varphi(0)=0$. In other terms, 
$X$ and $Y$ belong to the same scale family of distributions, with 
$$
 F_X(x)=F_Y\Big(\frac{\alpha_Y}{\alpha_X}\,x\Big), \qquad 
 \forall x\in\mathbb{R}. 
$$
\begin{proposition}\label{prop:propor}
The random variables $X^L=\Psi^L(0,X)$ and $Y^L=\Psi^L(0,Y)$ are identically distributed if, 
and only if, $X$ and $Y$ have proportional quantile functions as specified in (\ref{eq:propquant}). 
\end{proposition}
\begin{proof}
Since the distribution function of $X^L$ is given by (\ref{eq:defLp}), the proof thus follows.  
\end{proof}
\begin{proposition}\label{prop:propor2}
If the quantile functions of $X$ and $Y$ are proportional as expressed in Eq.\ (\ref{eq:propquant}), 
with $\alpha_Y>\alpha_X>0$, then $X\leq_{st}Y$. Moreover, $Z^L=\Psi^L(X,Y)$ is identically distributed 
to $X^L$ and $Y^L$, with density $f_{Z^L}(u)=\varphi(u)/{\eta}$, $0< u<1$, and the following equality holds:
\begin{equation}
 E[\{g(1)-g(U)\}\, \varphi'(U)]=\eta \,E\left[g'(Z^L)\right],
 \label{eq:theorprop}
\end{equation}
where $U$ is uniformly distributed in $(0,1)$, and $g:(0,1)\to \mathbb{R}$ is a differentiable function. 
\end{proposition}
\begin{proof}
The proof follows from Propositions \ref{prop:1} and \ref{prop:propor}, and from Theorem \ref{theor:3}. 
\end{proof}
\par 
We remark that the variables considered in Example \ref{ex:1} satisfy the assumptions of 
Proposition \ref{prop:propor2}. Other cases are shown hereafter. 
\begin{example}\label{example:1prop}\rm
(i) Let $X$ and $Y$ have the following distribution functions, with $\beta>0$:
$$
 F_X(x)=\Big(\frac{x}{\alpha_X}\Big)^{\beta}, \qquad 0\leq x\leq \alpha_X,
 \qquad   
 F_Y(x)=\Big(\frac{x}{\alpha_Y}\Big)^{\beta}, \qquad 0\leq x\leq \alpha_Y.
$$ 
If $\alpha_Y>\alpha_X>0$, then the assumptions of Proposition \ref{prop:propor2} are satisfied. 
Hence, Eq.\  (\ref{eq:theorprop}) holds, with $\varphi(u)=u^{1/\beta}$, $0< u< 1$, and 
$\eta=\frac{\beta}{1+\beta}$, so that $Z^L=\Psi^L(X,Y)$ has density 
$f_{Z^L}(u)=\frac{\beta+1}{\beta}u^{1/\beta}$, $0< u<1$.  \\
(ii) Let $X$ and $Y$ have Pareto (Type I) distribution, with  
$$
 F_X(x)=1-\Big(\frac{\alpha_X}{x}\Big)^{\beta}, \qquad x\geq \alpha_X>0,
 \qquad  
 F_Y(x)=1-\Big(\frac{\alpha_Y}{x}\Big)^{\beta}, \qquad x\geq \alpha_Y>0,
$$ 
for $\beta>0$. If $\alpha_Y>\alpha_X>0$ and if $\beta>1$, then the hypotheses of 
Proposition \ref{prop:propor2} hold. Relation (\ref{eq:theorprop}) is thus fulfilled, with 
$\varphi(u)=(1-u)^{-1/\beta}$, $0< u< 1$, and $\eta=\frac{\beta}{\beta-1}$, by which 
$Z^L=\Psi^L(X,Y)$ has density $f_{Z^L}(u)=\frac{\beta-1}{\beta(1-u)^{1/\beta}}$, $0< u< 1$. 
\end{example}
%
\section{Applications}\label{section:appl}
Let us now analyse various applications of the results given in the previous section. 
\subsection{Classes of distributions}
Among the classes of probability distributions, wide attention is given to the following notions. 
Let $X$ be a nonnegative random variable; then 
\begin{itemize}
\item[(i)] $X$ is NBU (new better that used) $\Leftrightarrow$ $X_t\leq_{st} X$ for all $t\geq 0$, i.e.\ 
$\barF(s)\barF(t) \geq  \barF(s+t)$ for all $s \geq 0$ and $t \geq 0$, 
\item[(ii)] $X$ is IFR (increasing failure rate) $\Leftrightarrow$ $X_t\leq_{st} X_s$ for all $t\geq s\geq 0$, i.e.\ 
$\barF$ is logconcave, 
\end{itemize}
where $X_t:=[X-t\,|\,X>t]$, for $t\geq 0$. The above notions can be expressed also in terms of the 
quantiles. Consider the residual of $X$ evaluated at $Q(p)$, i.e.  
\begin{equation}
 X_{Q(p)}=[X-Q(p)\,|\,X>Q(p)]\qquad \hbox{for $0<p<1$.}
 \label{eq:defXQp}
\end{equation}
In risk theory $X_{Q(p)}$ describes the losses exceeding $Q(p)$. Indeed, in a population of 
losses distributed as $X$, then $X_{Q(p)}$ denotes the residual of a loss whose level is 
equal to the  $p$th quantile, for $0<p<1$. 
If $X$ has a strictly increasing quantile function $Q(p)$, then
\begin{itemize}
\item[(i)] $X$ is NBU $\Leftrightarrow$ $\barF(Q(p)+Q(r))\leq (1-p)(1-r)$ for all $p,r\in(0,1)$,
\item[(ii)] $X$ is IFR $\Leftrightarrow$ $\frac{\barF(Q(s)+Q(p))}{\barF(Q(s)+Q(r))}\leq \frac{1-p}{1-r}$ 
for all $0<p< r<1$ and $0<s<1$.
\end{itemize}
\par
It is worth noting that comparisons of variables defined as in (\ref{eq:defXQp}) allow to define 
the lr-order of the dispersion type (see Belzunce {\em et al.}\ \cite{BeHuKh2003}). 
We recall that, if $X\in{\cal D}$, the {\em conditional value-at-risk} of $X$  is given by
(see Belzunce {\em et al.}\ \cite{BePiRuSo12}, or Denuit {\em et al.}\ \cite{De2005}): 
\begin{equation}
 CVaR[X;p] := E[X_{Q(p)}]=E[X-Q(p)\,|\,X>Q(p)]
 = \frac{1}{1-p}\int_{Q(p)}^{+\infty}\barF(y)\,{\rm d}y, 
 \qquad 0<p<1,
 \label{eq:defCVaR}
\end{equation}
with $CVaR[X;0]=E[X]$. 
(Note that in the literature some authors give different definitions for the conditional value-at-risk.) 
In the context of reliability theory the function given in (\ref{eq:defCVaR}) is also named 
`mean residual quantile function', since  
$$
 CVaR[X;p]={\rm mrl}(Q(p))=\frac{1}{1-p}\int_p^1[Q(t)-Q(p)]\,{\rm d}t,\qquad  0<p<1,
$$ 
where ${\rm mrl}(t)=E[X_t]=E[X-t\,|\,X>t]$ is the {\em mean residual life\/} of a lifetime $X$ evaluated at 
age $t\geq 0$. Furthermore, we point out that the conditional value-at-risk is also related to the 
{\em right spread function\/} of $X$ through the following identity: $CVaR[X;p]=S_X^+(p)/(1-p)$ 
(see, for instance, Fernandez-Ponce {\em et al.\/} \cite{FeKoMu98} for several results on $S_X^+(p)$). 
Finally, the integral in the right-hand-side of (\ref{eq:defCVaR}) is known as the `excess wealth transform', 
and plays an essential role in the excess wealth order (cf.\ Section 3.C of \cite{ShSh07}). 
\begin{remark}\rm
Given two nonnegative random variables $X$ and $Y$, one has 
$\frac{1}{\alpha_X}CVaR[X;p] = \frac{1}{\alpha_Y} CVaR[Y;p]$ for all $p\in(0,1)$ if and only if 
the proportional quantile functions model holds as specified in (\ref{eq:propquant}).  
\end{remark}
\par
Let us now provide a result involving NBU random variables.
\begin{proposition}\label{prop:NBU}
Let $X\in {\cal D}$ be NBU and such that $CVaR[X;p]<E[X]$ for all $p\in(0,1)$. 
If $g:(0,1)\to \mathbb{R}$ is a differentiable function, and if 
$U$ is uniformly distributed in $(0,1)$, then for all $p\in(0,1)$ we have
\begin{equation}
 E[\{g(1)-g(U)\} \{q(U)-q(1-(1-p)(1-U))(1-p)\}]=E[g'(Z^L)] \{E[X]-CVaR[X;p]\},
 \label{eq:relXNBU}
\end{equation}
where $Z^L:=\Psi^L(X_{Q(p)},X)$ has density 
$$
  f_{Z^L}(u)=\frac{Q(u)+Q(p)-Q(1-(1-p)(1-u))}{E[X]-CVaR[X;p]}, 
 \qquad 0\leq u<1.
$$
\end{proposition}
\begin{proof}
The mean of $X_{Q(p)}$ is expressed in (\ref{eq:defCVaR}), whereas due to (\ref{eq:defXQp}) its quantile 
function and quantile density for $p\in(0,1)$ and $u\in(0,1)$ are given respectively by 
\begin{equation}
 Q_{X_{Q(p)}}(u)=Q(1-(1-u)(1-p))-Q(p), 
 \label{eq:esprQXQpu}
\end{equation}
and 
$$  
 q_{X_{Q(p)}}(u)=q(1-(1-u)(1-p))(1-p).
$$
The assertion then follows from Proposition \ref{prop:1} and Theorem \ref{theor:3}. 
\end{proof}
\par 
We remark that the quantile function given in (\ref{eq:esprQXQpu}) is often used to model reliability 
data, since it represents the $u$th percentile residual life expressed in terms of quantile, as shown in 
Eq.\ (2.7) of Unnikrishnan Nair and Sankaran \cite{NaSa2009}. 
\par 
In the line of Proposition \ref{prop:NBU}  we now provide a similar result for IFR random variables.
\begin{proposition}\label{prop:agingIFR}
Let $X\in {\cal D}$ be IFR and such that $CVaR[X;p]$ 
is strictly decreasing for $0<p<1$. Then, for all $0<r<p<1$ we have
\begin{multline}
E[\{g(1)-g(U)\}\{q(1-(1-r)(1-U))(1-r)-q(1-(1-p)(1-U))(1-p)\}] \\
=E[g'(Z^L)]\{CVaR[X;r]-CVaR[X;p]\},
\nonumber
\end{multline}
where $U$ is uniformly distributed in $(0,1)$, and $Z^L:=\Psi^L(X_{Q(p)},X_{Q(r)})$ has density 
\begin{equation}
  f_{Z^L}(u)=\frac{Q(1-(1-r)(1-u))-Q(r)-Q(1-(1-p)(1-u))+Q(p)}{CVaR[X;r]-CVaR[X;p]}, 
 \qquad 0< u<1.
 \label{eq:desfZLIFR}
\end{equation}
\end{proposition}
\begin{proof}
Since $X$ is IFR, we have $X_{Q(p)}\leq_{st} X_{Q(r)}$ for all $0<r<p<1$. The proof thus 
proceeds similarly as  Proposition \ref{prop:NBU}.
\end{proof}
%
%
\begin{example}\rm
Let $X=\max\{T_1,\ldots,T_N\}$, where $T_i$, $i\geq 1$, are independent, identically ${\rm Exp}(\lambda)$-distributed 
random variables, and where $N$ is a geometric random variable independent of $X_i$, $i\geq 1$, and with 
parameter $\delta\in (0,1)$. Then, $X$ is IFR (see Example 7.2 of Ross {\em et al.}\ \cite{RoShZh05}). 
Its quantile function and quantile density are respectively given by
$$
 Q(u)=\frac{1}{\lambda}\ln \left(\frac{\delta+(1-\delta)u}{\delta(1-u)}\right), 
 \qquad 
 q(u)=\left[\lambda (1-u)(\delta+(1-\delta)u)\right]^{-1},
 \qquad 0<u<1. 
$$
From (\ref{eq:defCVaR}) we see that the conditional value-at-risk of $X$ is: 
$$ 
 CVaR[X;p] = - \frac{ \ln (p+(1-p)\delta)}{\lambda(1-p)(1-\delta)}, 
 \qquad 0<p<1. 
$$
We note that $CVaR[X;p]$ is strictly decreasing in $p\in(0,1)$. Hence, from 
Proposition \ref{prop:agingIFR} we have, for $0<r<p<1$, 
\begin{multline}
E\left[\frac{g(1)-g(U)}{1-U}\,\left\{
\frac{1}{U(1-r)(1-\delta)+r(1-\delta)+\delta}
-\frac{1}{U(1-p)(1-\delta)+p(1-\delta)+\delta}\right\}\right] \\
=E[g'(Z^L)]\,\frac{1}{1-\delta} \left\{ \frac{ \ln (p+\delta (1-p))}{1-p}
- \frac{ \ln (r+\delta (1-r))}{1-r}\right\},
\nonumber
\end{multline}
where $U$ is uniformly distributed in $(0,1)$. The density of $Z^L=\Psi^L(X_{Q(p)},X_{Q(r)})$ 
can be obtained from (\ref{eq:desfZLIFR}) and the above given expressions. 
\end{example}
\par
We remark that Propositions \ref{prop:NBU} and \ref{prop:agingIFR} provide identities holding for specific 
ranges of the involved parameters. However, a `local version' of such results can be easily stated under 
mild assumptions. For instance Eq.\ (\ref{eq:relXNBU}) holds for a fixed $p\in(0,1)$, provided that  
$X_{Q(p)}\leq _{st} X$ and $CVaR[X;p]<E[X]$ for such fixed $p$. An example in which these conditions  
hold for some $p\in(0,1)$ is provided when $X$ is the maximum of two independent exponential distributions 
with unequal parameters, whose distribution is IFRA (increasing failure rate in average) and thus NBU, 
but not IFR (see, for instance, Klefsj\"o \cite{Kl83}). 
\subsection{Risks}
When comparing risks, the quantile function $Q(p)$ of $X$ plays a very important role. 
In fact, in this context it is known as {\em value-at-risk\/} and is denoted by $VaR[X;p]\equiv Q(p)$, $0<p<1$. 
However, to avoid discrepancies we adopt the notation $Q(p)$. Given a nonnegative random variable $X$ 
with finite mean, we define  
\begin{equation}
 \widetilde X_p:=\left\{\left.\frac{X-Q(p)}{Q(p)}\right|X>Q(p)\right\}
 \label{eq:defX_p}
\end{equation}
for all $p\in S_X:=\{u\in(0,1): Q(u)>0\}$. The random variable $\widetilde X_p$ is useful to compare 
risks of different nature, and can be viewed as {\em proportional conditional value-at-risk\/} because 
it measures the conditional upper tail from $Q(p)$ on, but proportional to $Q(p)$. Moreover, from 
(\ref{eq:defXQp}) and (\ref{eq:defX_p}) we have $\widetilde X_p=X_{Q(p)}/Q(p)$ for all $p\in S_X$. 
Hence, Eqs.\ (\ref{eq:defCVaR}) and (\ref{eq:defX_p}) yield 
\begin{equation}
 E[\widetilde X_p]=\frac{CVaR[X;p]}{Q(p)}, \qquad p\in S_X.
 \label{eq:medX_p}
\end{equation}
\par
In this case, we can consider conditions similar to NBU and IFR properties which are defined in 
terms of (\ref{eq:defX_p}):
\begin{itemize}
\item[(i)] $\widetilde X_p\leq_{st} X$ for all $0<p<1$ $\Leftrightarrow$ 
$\barF((1+x)Q(p))\leq \barF(x)(1-p)$ for all $0<p<1$ and $x\geq 0$.
\item[(ii)] $\widetilde X_p\leq_{st} \widetilde X_r$ for all $0<r<p<1$ $\Leftrightarrow$ 
$\frac{\barF((1+x)Q(p))}{\barF((1+x)Q(r))}\leq \frac{1-p}{1-r}$ for all $0<r<p<1$ and $x\geq 0$.
\end{itemize}
\begin{proposition}\label{prop:risk1}
Let $X\in {\cal D}$ be  such that $\widetilde X_p\leq_{st} X$ and $CVaR[X;p]< E[X]\,Q(p)$ for all $p\in S_X$. 
If $g:(0,1)\to \mathbb{R}$ is a differentiable function and if $U$ is uniformly distributed in $(0,1)$, 
then for all $p\in S_X$ we have
\begin{multline}
 E\left[\{g(1)-g(U)\}\left\{q(U)-\frac{1}{Q(p)}\,q(1-(1-p)(1-U))(1-p)\right\}\right] \\
 =E[g'(Z^L)]\left\{E[X]-\frac{1}{Q(p)}\,CVaR[X;p]\right\}
 \nonumber,
\end{multline}
where $Z^L=\Psi^L(\widetilde X_p,X)$ is a random variable with density function  
\[
f_{Z^L}(u)=\frac{(1+Q(u))Q(p)-Q(1-(1-p)(1-u))}{E[X]Q(p)-CVaR[X;p]}, \qquad 0<u<1.
\]
\end{proposition}
\begin{proof}
Since the mean of $\widetilde X_p$ is (\ref{eq:medX_p}), and the quantile function and quantile 
density are given respectively by 
$$
 Q_{\widetilde X_p}(u)=\displaystyle\frac{1}{Q(p)}\,Q(1-(1-p)(1-u))-1,
 \qquad 
 q_{\widetilde X_p}(u)=\displaystyle\frac{1}{Q(p)}\,q(1-(1-p)(1-u))(1-p) 
$$
for each $p\in S_X$ and $0<u<1$, the proof follows from Proposition \ref{prop:1} and 
Theorem \ref{theor:3}.
\end{proof}
\par
An extension of Proposition \ref{prop:risk1} to a more general case is given hereafter. 
The proof is analogous, and then is omitted. 
\begin{proposition}\label{prop:risk2}
Let $X\in {\cal D}$ be such that 
$\widetilde X_p\leq_{st} \widetilde X_r$ and $CVaR[X;p]/Q(p)$ is strictly decreasing for 
all $p\in S_X$.  If $g:(0,1)\to \mathbb{R}$ is a differentiable function and if $U$ is 
uniformly distributed in $(0,1)$, then for all $r, p\in S_X$ such that $0<r<p<1$ we have 
\begin{multline}
 E\left[\{g(1)-g(U)\}\left\{\frac{q(1-(1-r)(1-U))(1-r)}{Q(r)}
 -\frac{q(1-(1-p)(1-U))(1-p)}{Q(p)}\right\}\right]\\
 =E[g'(Z^L)]\left\{\frac{CVaR[X;r]}{Q(r)}-\frac{CVaR[X;p]}{Q(p)}\right\},
 \nonumber
\end{multline}
where $Z^L=\Psi^L(\widetilde X_p,\widetilde X_r)$ is a random variable having density function  
\begin{equation}
f_{Z^L}(u)=\displaystyle
 \frac{Q(p)\,Q(1-(1-r)(1-u))-Q(r)\,Q(1-(1-p)(1-u))} 
 {Q(p)\,CVaR[X;r]- Q(r)\,CVaR[X;p]}, \qquad 0<u<1.
 \label{eq:densZLproprisk2}
\end{equation}
\end{proposition}
\begin{example}\rm
Let $X$ have Rayleigh distribution, $F(x)=1-e^{-\alpha x^2}$, $x\geq 0$, with $\alpha>0$. Hence, 
the quantile function and the quantile density of $X$ are: 
$$
 Q(u)=\sqrt{|\ln(1-u)|/\alpha},
 \qquad
 q(u)=\sqrt{\alpha}\left[2(1-u)\sqrt{|\ln (1-u)|}\right]^{-1}, 
 \qquad 0<u<1. 
$$
Due to  (\ref{eq:defCVaR})  the conditional value-at-risk of $X$ is given by
$$ 
 CVaR[X;p] = \frac{\sqrt{\pi} \,{\rm erfc}\left[\sqrt{|\ln (1-p)|}\right]}{2\sqrt{\alpha}\,(1-p)}, 
 \qquad 0<p<1,
$$
where ${\rm erfc}[\cdot]$ denotes the complementary error function. It is not hard to verify that $CVaR[X;p]$ 
is strictly decreasing in $p\in(0,1)$. Moreover, since $X$ is IFR, from Proposition \ref{prop:agingIFR} 
we have, for all $0<r<p<1$,
\begin{multline}
 E\left[\frac{g(1)-g(U)}{1-U} 
 \left\{\frac{1}{\sqrt{|\ln [(1-r)(1-U)]|}}
 -\frac{1}{\sqrt{|\ln [(1-p)(1-U)]|}}\right\}\right] \\
=E[g'(Z^L)]\frac{\sqrt{\pi}}{ \alpha}\left\{
\frac{{\rm erfc}\left[\sqrt{|\ln (1-r)|}\right]}{1-r}
-\frac{{\rm erfc}\left[\sqrt{|\ln (1-p)|}\right]}{1-p}
\right\},
\nonumber
\end{multline}
where $g:(0,1)\to \mathbb{R}$ is a differentiable function, $U$ is uniformly distributed in $(0,1)$, 
and the density of $Z^L=\Psi^L(X_{Q(p)},X_{Q(r)})$ can be obtained from (\ref{eq:desfZLIFR}).  
Furthermore,  $CVaR[X;p]/Q(p)$ is strictly decreasing in $p\in(0,1)$. Hence, 
Proposition \ref{prop:risk2} yields the following indentity, for $0<r<p<1$, 
\begin{multline}
 E\left[\frac{g(1)-g(U)}{1-U} 
 \left\{\frac{1}{\sqrt{\ln (1-r)\ln [(1-r)(1-U)]}}
 -\frac{1}{\sqrt{\ln (1-p)\ln [(1-p)(1-U)]}}\right\}\right] \\
=E[g'(Z^L)]\frac{\sqrt{\pi}}{\alpha}\left\{
\frac{{\rm erfc}\left[\sqrt{|\ln (1-r)|}\right]}{(1-r)\sqrt{|\ln (1-r)|}}
-\frac{{\rm erfc}\left[\sqrt{|\ln (1-p)|}\right]}{(1-p)\sqrt{|\ln (1-p)|}}
\right\}.
\nonumber
\end{multline}
In this case, owing to  (\ref{eq:densZLproprisk2}), for $0<u<1$ the density of 
$Z^L=\Psi^L(\widetilde X_p,\widetilde X_r)$ is 
\begin{equation}
 f_{Z^L}(u)=\frac
 {2(1-p)(1-r)\left\{\sqrt{\ln(1-r)\ln[(1-p)(1-u)]}-\sqrt{\ln(1-p)\ln[(1-r)(1-u)]}\right\}}
 {\sqrt{\pi}\left\{ (1-r)\,{\rm erfc}\left[\sqrt{|\ln (1-p)|}\right] \sqrt{|\ln (1-r)|}
 -(1-p)\,{\rm erfc}\left[\sqrt{|\ln (1-r)|}\right] \sqrt{|\ln (1-p)|}\right\}}.
 \label{eq:densfZLUrisk}
\end{equation}
We remark that such density does not depend on $\alpha$. Some plots of (\ref{eq:densfZLUrisk}) 
are given in Figure \ref{fig:appl2}.
\end{example}
%
\begin{figure}[t] 
\begin{center}
\epsfxsize=8.5cm
\epsfbox{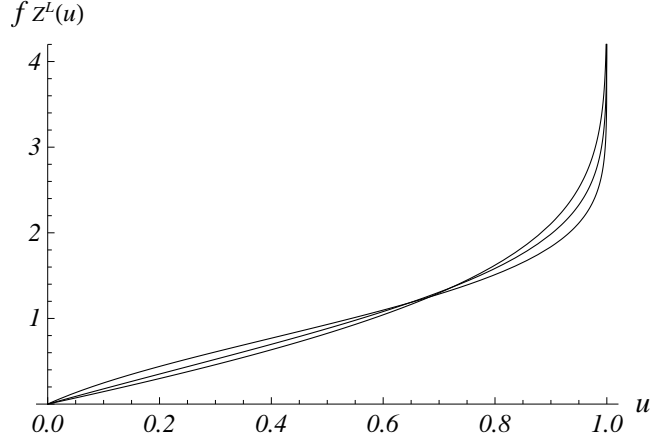}
\caption{Density (\ref{eq:densfZLUrisk}), for 
$(r,p)=(0.1,0.3), (0.4,0.6), (0.7,0.9)$ (from top to bottom near the origin).}
\label{fig:appl2}
\end{center}
\end{figure}
%
\subsection{A model involving the average value-at-risk}\label{app3}
Let $X\in {\cal D}$ have quantile function $Q(p)$. One can introduce a new 
family of random variables $X^*_v$, $0<v<1$, having distribution function 
\[
F^*_v(x)=\frac{Q(v x)}{Q(v)}, \qquad 0<x<1.
\]
If $X$ represents a risk, the {\em average value-at-risk\/} of $X$ is defined as 
\begin{equation}
 AVaR[X;v]=E[X\,|\,X\leq Q(v)]=\frac{1}{v}\int_0^v Q(u)\,{\rm d}u, \qquad 0<v<1.
 \label{eq:defCPE}
\end{equation} 
Note that the risk measure given in (\ref{eq:defCPE}) represents the conditional expected loss 
given that the loss $X$ is less than its value-at-risk.  
See Chapter 6 of Rachev {\em et al.}\ \cite{RaStFa2011} for results, properties 
and applications in mathematical finance of $AVaR[X;v]$. The average value-at-risk plays a 
significant role also in stochastic orders of interest in risk theory (see Jewitt \cite{Je89}). 
In addition, $AVaR[X;v]$ can be viewed as the L-moment of order 1 of the reversed quantile function 
(see Unnikrishnan Nair and Vineshkumar \cite{NaVi2010}). 
\par
We can easily show that the mean of $X^*_v$ can be expressed in terms of $AVaR[X;v]$ as 
\begin{equation}
 E[X^*_v]=1-\frac{1}{Q(v)}\,AVaR[X;v], \qquad 0<v<1.
 \label{eq:medXv}
\end{equation} 
Then, we have the following result, where $f$ denotes the density of $X$.
\begin{proposition}\label{application3}
Let $X\in {\cal D}$ be such 
that $X^*_v\leq_{st} X^*_w$ for $0<v<w<1$, and  $\frac{1}{Q(v)}\,AVaR[X;v]$ 
is strictly decreasing in $v\in (0,1)$. If $g:(0,1)\to \mathbb{R}$ is a differentiable function and if $U$ 
is uniformly distributed in $(0,1)$, then for all $0<v<w<1$ we have
\begin{multline}
E\left[\{g(1)-g(U)\}\left\{\frac{1}{w}f(Q(w)\,U)\,Q(w)-\frac{1}{v}f(Q(v)\,U)\,Q(v)\right\}\right] \\
=E[g'(Z^L)]\left\{\frac{AVaR[X;v]}{Q(v)}-\frac{AVaR[X;w]}{Q(w)}\right\},
\nonumber
\end{multline}
where $Z^L=\Psi^L(X^*_v,X^*_w)$ is a random variable with density function  
\[
 f_{Z^L}(u)=\frac{\frac{1}{w}\,F(Q(w)\,u)-\frac{1}{v}\,F(Q(v)\,u)}
 {\frac{1}{Q(v)}\,AVaR[X;v]-\frac{1}{Q(w)}\, AVaR[X;w]},
\qquad 0<u<1.
\]
\end{proposition}
\begin{proof}
We recall that the mean of $X^*_v$ is (\ref{eq:medXv}). Moreover, its quantile function and quantile 
density for $v\in (0,1)$ are respectively given by 
\begin{equation}
 Q^*_v(u)=\frac{1}{v}F(Q(v)\,u), \qquad 
 q^*_v(u)=\frac{1}{v}f(Q(v)\,u)\,Q(v), \qquad 0<u<1.
 \label{eq:Qq*}
\end{equation}
The proof thus follows applying Proposition \ref{prop:1} and Theorem \ref{theor:3}. 
\end{proof}
\par
Let us now provide an equivalent condition for $X^*_v\leq_{st} X^*_w$, $0<v<w<1$, which was considered 
in Proposition \ref{application3}. 
\begin{proposition}\label{st-application3}
Let $X\in {\cal D}$. Then, $X^*_v\leq_{st} X^*_w$ for $0<v<w<1$ if, and only if, $x\,\tau(x)$ is decreasing 
for $x>0$, where $\tau(x)=\frac{{\rm d}}{{\rm d}x}\log F(x)$ 
is the reversed hazard rate function of $X$.
\end{proposition}
\begin{proof}
Given $0<v<w<1$, we have $X^*_v\leq_{st} X^*_w$ if, and only if, $Q_v^*(u)\leq Q_w^*(u)$ for all $u\in(0,1)$. 
Hence, due to the first of (\ref{eq:Qq*}), this property is equivalent to the following condition:
\begin{equation}
 \frac{F(x)}{F(\frac{x}{u})}\text{ \ is increasing in }x>0 \text{  for all }u\in(0,1).
 \label{st1-application3}
\end{equation}
Since $F(x)=\exp\{-\int^{+\infty}_x \tau(z)\,{\rm d}z\}$, $x>0$, condition (\ref{st1-application3}) holds if, and only if, 
$$
 \int^{x/u}_x \tau(z)\,{\rm d}z \text{  \ is decreasing in }x>0 \text{  for all }u\in(0,1).
$$
By differentiation we see that this condition is satisfied if, and only if, $u\,\tau(x)\geq \tau(\frac{x}{u})$ for all $x>0$ 
and $u\in(0,1)$. Finally, by setting $u=x/y$, with $x<y$, we obtain  $x\,\tau(x)\geq y\,\tau(y)$, this giving the proof. 
\end{proof}
\begin{example}\rm
Let $X$ have distribution function $F(x)=\exp\{-cx^{-\gamma}\}$, $x>0$, with $c,\gamma>0$. 
In this case, $\tau(x)=c\gamma x^{-(\gamma+1)}$ and thus $x\,\tau(x)$ is decreasing. Moreover, from 
(\ref{eq:defCPE}) we have $AVaR[X;v]=c^{1/\gamma} \Gamma(1-1/\gamma, -\ln v)$, where 
$\Gamma (\cdot,\cdot)$ is the incomplete gamma function. Hence, recalling Proposition \ref{st-application3}, 
the assumptions of Proposition \ref{application3} are satisfied. For instance, by setting for simplicity 
$\gamma=1$, for all $0<v<w<1$ we have 
$$
 E\left[\frac{g(1)-g(U)}{U^2}\left\{ w^{1/U-1} |\ln w|  - v^{1/U-1} |\ln v| \right\}\right]
 =E[g'(Z^L)]\left\{ \frac{1}{v} \, {\rm li}(v)\,\ln v  - \frac{1}{w} \, {\rm li}(w)\,\ln w\right\},
$$
where ${\rm li}(x)=\int_0^x (\ln t)^{-1}\,{\rm d}t$ is the logarithmic integral function. 
Moreover $Z^L=\Psi^L(X^*_v,X^*_w)$ has density function  
\begin{equation}
 f_{Z^L}(u)=\frac{w^{1/u-1}-v^{1/u-1}}
 {\frac{1}{v} \, {\rm li}(v)\,\ln v  - \frac{1}{w} \, {\rm li}(w)\,\ln w},
 \qquad 0<u<1.
 \label{eq:densfli}
\end{equation}
Some plots of $f_{Z^L}(u)$ are given in Figure \ref{fig:appl3}.  
\end{example}
%
\begin{figure}[t] 
\begin{center}
\epsfxsize=7.5cm
\epsfbox{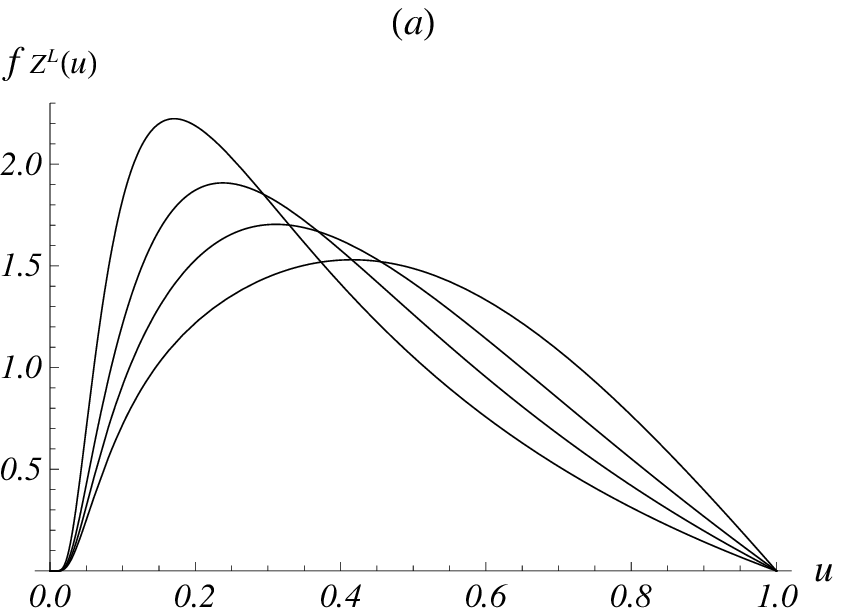}
$\;$
\epsfxsize=7.5cm
\epsfbox{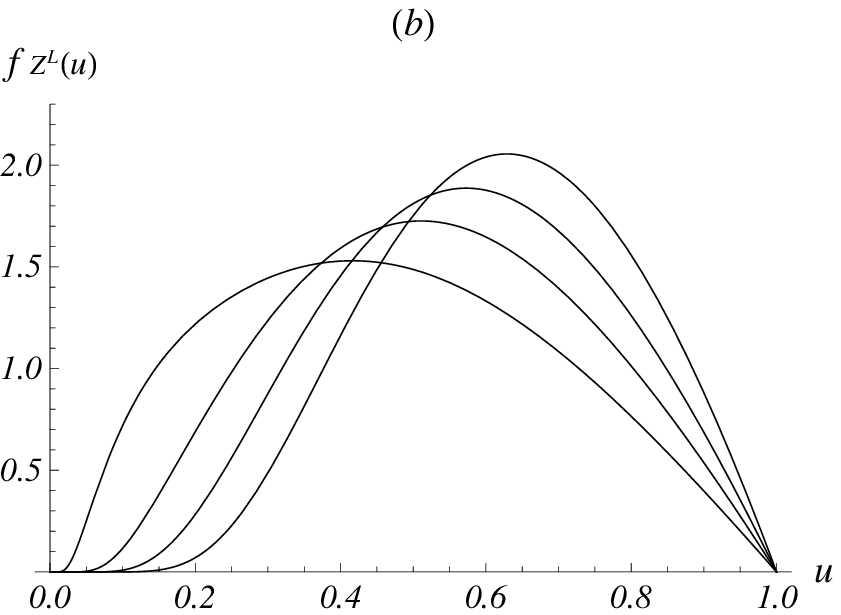} 
\caption{Density (\ref{eq:densfli}), for 
(a) $w=0.9$, $v=0.1, 0.3, 0.5, 0.7$ (from bottom to top near the origin), 
and 
(b) $v=0.1$, $w=0.3, 0.5, 0.7, 0.9$ (from top to bottom near $u=1$).}
\label{fig:appl3}
\end{center}
\end{figure}
%
\par 
Hereafter we show an example of distribution function that does not satisfy the conditions of 
Proposition \ref{st-application3}. 
\begin{counterexample}\label{counterexample}\rm
Let $X$ a random variable with distribution function and reversed hazard rate given respectively by 
(cf.\ Section 2 of Block {\em et al.}\ \cite{BlSaSi98}) 
$$
F(x) = \left\{
\begin{array}{l l}
\exp\{-1-\frac{1}{x}\}, & 0<x< 1\\
\exp\{\frac{x^2-5}{2}\}, & 1\leq x< 2\\
\exp\{-\frac{1}{x}\}, & 2\leq x< \infty,
\end{array}
\right.
\qquad 
 \tau(x) = \left\{
\begin{array}{l l}
\frac{1}{x^2}, & 0<x< 1\\
x, & 1\leq x< 2\\
\frac{1}{x^2}, & 2\leq x< \infty.
\end{array}
\right.
$$
It is not hard to see that $x\, \tau(x)$ is not decreasing, and thus Proposition \ref{application3} cannot 
be applied.
\end{counterexample}
%
\subsection{A model based on increasing variables}
Let $X\in {\cal D}$ have density $f(x)$ and quantile function $Q(p)$. We now consider
a random variable $\widehat X_v$, $v\in (0,1)$,  with support $(0,v)$ and distribution function  
\begin{equation}
 \widehat F_v(x):=\frac{Q(x)}{Q(v)},\qquad 0\leq x\leq v.
 \label{eq:defwFvx}
\end{equation}
From (\ref{eq:defwFvx}) it is not hard to see that 
\begin{equation}
 E[\widehat X_v]=v\left[1-\frac{1}{Q(v)}\,AVaR[X;v]\right], \qquad 0<v<1,
 \label{eq:medwFvx}
\end{equation}
where $AVaR[x;v]$ is the average value-at-risk defined in (\ref{eq:defCPE}). 
\begin{proposition}\label{application4}
Let $X\in{\cal D}$. If $g:(0,1)\to \mathbb{R}$ is a differentiable function and if $U$ is 
uniformly distributed in $(0,1)$, then for all $0<v<w<1$ we have 
\begin{multline}
 E[\{g(1)-g(U)\} \{f[Q(w)\,U]\,Q(w)-f[Q(v)\,U]\,Q(v)\}] \\
 =E[g'(Z^L)]\left\{w\left(1-\frac{AVaR[X;w]}{Q(w)}\right)-v\left(1-\frac{AVaR[X;v]} {Q(v)}\right)\right\},
 \nonumber
\end{multline}
where $Z^L=\Psi^L(\widehat X_v,\widehat X_w)$ is a random variable with density function  
\[
f_{Z^L}(u)=\frac{F[Q(w)\,u]-F[Q(v)\,u]}
{w\left(1-\frac{1}{Q(w)}\,AVaR[X;w]\right)-v\left(1-\frac{1}{Q(v)}\,AVaR[X;v]\right)}, 
\qquad 0<u<1.
\]
\end{proposition}
\begin{proof}
From (\ref{eq:defwFvx}) it is easy to see that the quantile function and  the quantile density of $\widehat X_v$, 
$0<v<1$, are respectively given by 
$$
 \widehat Q_{v}(u)=F[Q(v)\,u], \qquad 
 \widehat q_{v}(u)=f[Q(v)\,u]\,Q(v), \qquad 0<u<1.
$$
Note that $\widehat X_v$ is stochastically increasing in $v\in(0,1)$, since 
$\widehat X_v\leq_{st} \widehat X_w$ for all $0<v<w<1$. Moreover, the given assumptions ensure 
that the mean (\ref{eq:medwFvx}) is strictly increasing in $v\in(0,1)$. 
The proof thus follows from Proposition \ref{prop:1} and Theorem \ref{theor:3}.
\end{proof}
\begin{example}\rm
Let $X$ be an exponentially distributed random variable  with parameter $\lambda>0$. Then, 
for all $0<v<1$ and $0<u<1$ we have 
$$
 E[\widehat X_v]=\frac{v}{\ln(1-v)}+1, 
 \qquad 
 \widehat Q_{v}(u)=1-(1-v)^u, 
 \qquad 
 \widehat q_{v}(u)=-(1-v)^u \ln(1-v).
$$
Note that the above functions do not depend on $\lambda$. 
Therefore, from Proposition \ref{application4} it follows 
\begin{multline}
E\left[\{g(1)-g(U)\}\left\{(1-v)^U \ln(1-v)-(1-w)^U \ln(1-w)\right\}\right]\\
=E[g'(Z^L)]\left\{\frac{w}{\ln(1-w)}-\frac{v}{\ln(1-v)}\right\},
\end{multline}
for all $0<v<w<1$. Here $Z^L=\Psi^L(\widehat X_v,\widehat X_w)$ is a random variable with 
density function  
\begin{equation}
f_{Z^L}(u)=\frac{(1-v)^u-(1-w)^u}{\frac{w}{\ln(1-w)}-\frac{v}{\ln(1-v)}}, 
\qquad 0<u<1.
 \label{eq:densfZLUesesp}
\end{equation}
See Figure \ref{fig:appl4} for some plots of $f_{Z^L}(u)$. 
\end{example}
%
\begin{figure}[t] 
\begin{center}
\epsfxsize=8.5cm
\epsfbox{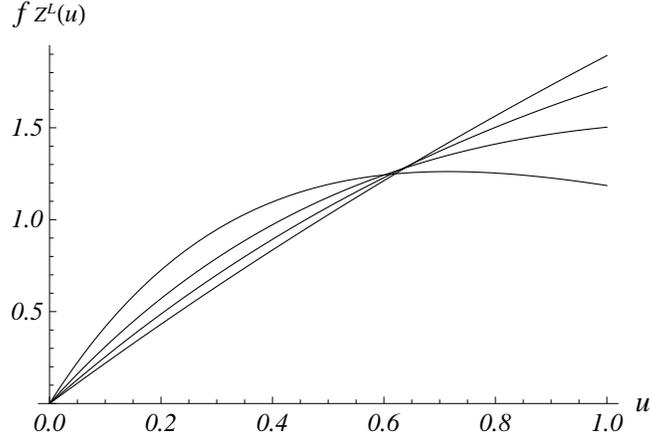}
\caption{Density (\ref{eq:densfZLUesesp}), for 
$(v,w)=(0.1,0.2), (0.3, 0.4), (0.5,0.6), (0.7,0.8)$ (from bottom to top near the origin).}
\label{fig:appl4}
\end{center}
\end{figure}
%
\subsection{Concluding remarks} 
In our view, the main issues of this paper are given 
in Proposition \ref{prop:1} and Theorem \ref{theor:3}. The first result allows us to construct new probability 
densities with support $(0,1)$ starting from suitable pairs of stochastically ordered random variables. 
The second result is useful to obtain equalities involving uniform-$(0,1)$ distributions and quantile 
functions. The cases treated in this section give only a partial view of the potentiality of   
Theorem \ref{theor:3}. Indeed, we considered some special cases in which the random variables 
$X$ and $Y$ involved in Theorem \ref{theor:3} belong to the same family of distributions. 
Other useful applications are likely to be developed under various choices of such variables, and 
specific selection of function $g$. This will be the object of future research. 
\section*{Acknowledgements} 
The research of A.\ Di Crescenzo and B.\ Martinucci  has been performed under partial support by 
GNCS-INdAM and Regione Campania (Legge 5). J.\ Mulero is supported by 
project MTM2012-34023, ``Comparaci\'on y dependencia en modelos probabil\'{\i}sticos con 
aplicaciones en fiabilidad y riesgos'', from Universidad de Murcia. 

%
%
\end{document}